\documentclass[11pt]{amsart}
\usepackage{latexsym}
\usepackage{amscd}
\usepackage{amssymb}
\usepackage[all]{xy}
\usepackage{amsmath, pb-diagram}
\DeclareMathAlphabet{\eusm}{OT1}{eusm}{m}{n}
\newtheorem{theorem}{Theorem}[section]

\newtheorem{remark}[theorem]{Remark}

\def\Ker{\mbox{Ker\/}}
\def\Im{\mbox{Im\/}}
\def\End{\mbox{End\/}}

\def\H{\mbox{Hom\/}}

\begin{document}
\title[Modules coinvariant under automorphisms of their projective covers]{Modules which are coinvariant under automorphisms of their projective covers}
\subjclass[2010]{16D40, 16D80.} 
\keywords{automorphism-invariant modules, automorphism-coinvariant modules, dual automorphism-invariant modules}
\author[Guil Asensio]{Pedro A. Guil Asensio}
\address{Departamento de Mathematicas, Universidad de Murcia, Murcia, 30100, Spain}
\email{paguil@um.es}
\author[Kesk\.{i}n T\"ut\"unc\"u]{Derya Kesk\.{i}n T\"ut\"unc\"u}
\address{Department of Mathematics, Hacettepe University, Ankara, 06800, Turkey}
\email{keskin@hacettepe.edu.tr}
\author[Kalebo\~{g}az]{Berke Kalebo\~{g}az}
\address{Department of Mathematics, Hacettepe University, Ankara, 06800, Turkey}
\email{bkuru@hacettepe.edu.tr}
\author[Srivastava]{Ashish K. Srivastava}
\address{Department of Mathematics and Computer Science, St. Louis University, St.
Louis, MO-63103, USA} \email{asrivas3@slu.edu}

\maketitle

\begin{abstract} 
In this paper we study modules coinvariant under automorphisms of their projective covers. We first provide an alternative, and in fact, a more succinct  and conceptual proof for the result that a module $M$ is invariant under automorphisms of its injective envelope if and only if given any submodule $N$ of $M$, any monomorphism $f:N\rightarrow M$ can be extended to an endomorphism of $M$ and then, as a dual of it, we show that over a right perfect ring, a module $M$ is coinvariant under automorphisms of its projective cover if and only if for every submodule $N$ of $M$, any epimorphism $\varphi: M\rightarrow M/N$ can be lifted to an endomorphism of $M$. 
\end{abstract}

\bigskip

\bigskip

\section{Introduction}

\bigskip

\noindent Modules invariant or coinvariant under automorphisms of their covers or envelopes have been recently introduced in \cite{GKS}. Recall that a class $\mathcal{X}$ of right modules over a ring $R$, closed under isomorphisms, is called an {\em enveloping class} if for any right $R$-module $M$, there exists a homomorphism $u:M\rightarrow X(M)$, with $X(M)\in \mathcal X$, such that any other morphism from $M$ to a module in $\mathcal X$ factors through $u$ and, moreover, whenever $u$ has a factorization $u=h\circ u$, then $h$ must be an automorphism. This morphism $u$ is called the {\em $\mathcal X$-envelope} of $M$. 
And this envelope is a {\em monomorphic envelope} if, in addition, $u$ is a monomorphism. 
Dually, $\mathcal X$ is called a {\em covering class} if for any right $R$-module $M$, there exists a homomorphism $p:X(M)\rightarrow M$ such that any other homomorphism from an object of $\mathcal X$ to $M$ factors through $p$ and moreover, whenever $p$ factors as $p=p\circ h$, $h$ must be an automorphism. 
This morphism $p$ is called the {\em $\mathcal X$-cover of $M$} and this cover is said to be an {\em epimorphic cover} if $p$ is an epimorphism. 

A module $M$ having a monomorphic $\mathcal X$-envelope $u:M\rightarrow X(M)$ (resp. epimorphic $\mathcal X$-cover $p:X(M)\rightarrow M$) is said 
to be {\em invariant under $\varphi$} 
(resp. {\em coinvariant under $\varphi$}), $\varphi:X(M)\rightarrow X(M)$, if there exists an endomorphism $f:M\rightarrow M$ such that $u\circ f=\varphi \circ u$ 
(resp. $f\circ p=p\circ\varphi$). 

A module $M$ having a monomorphic $\mathcal X$-envelope $u:M\rightarrow X(M)$ (resp. epimorphic $\mathcal X$-cover $p:X(M)\rightarrow M$) is said 
to be {\em $\mathcal X$-automorphism invariant} 
(resp. {\em $\mathcal X$-automomorphism coinvariant}) 
if $M$ is invariant (resp. coinvariant) under each automorphism $\varphi:X(M)\rightarrow X(M)$.

If a module $M$ is invariant (resp. coinvariant) under each endomorphism $\varphi:X(M)\rightarrow X(M)$, then $M$ is called {\em $\mathcal X$-endomorphism invariant} 
(resp. {\em $\mathcal X$-endomorphism coinvariant}).  

When $\mathcal X$ is the class of injective modules, $\mathcal X$-automorphism invariant modules are usually just called {\em automorphism-invariant} modules and $\mathcal X$-endomorphism invariant modules are called {\em quasi-injective} modules. When $\mathcal X$ is the class of projective modules, $\mathcal X$-automorphism coinvariant modules are called {\em automorphism-coinvariant} modules and $\mathcal X$-endomorphism coinvariant modules are called {\em quasi-projective} modules.

Automorphism-invariant modules have been studied extensively in \cite{AFT,DF,ESS,GS1,GS2, GS3,LZ,SS2}. On the other hand, it was proved in \cite{ESS} that a module $M$ is automorphism-invariant if and only if any monomorphism from a submodule $N$ of $M$ to $M$ extends to an endomorphism of $M$. The main goal of this note is to give a new and more conceptual proof of this result which allows to dualize it to automorphism-coinvariant modules.   

Throughout this note, all rings will be associative rings with identity and `module' will mean a unitary right module unless otherwise stated. We refer to \cite{AF, CLVW} for any undefined notion used along the text.  

\bigskip

\section{Main results}

\bigskip

\noindent We begin by noting an important structural result from \cite{GQS} which will be of crucial importance throughout. 

\begin{theorem} \label{struct} \cite{GQS} 
Let $\mathcal{X}$ be an enveloping (resp., covering) class of modules. If $u:M\rightarrow X$ is a monomorphic $\mathcal{X}$-envelope (resp., $p:X\rightarrow M$ is an epimorphic cover) of a module $M$ such that $M$ is $\mathcal{X}$-automorphism invariant (resp., $\mathcal{X}$-automorphism coinvariant) and $\End(X)/J(\End(X))$ is a von Neumann regular right self-injective ring and idempotents lift modulo $J(\End(X))$, then $\End(M)/J(\End(M))$ is also a von Neumann regular ring and idempotents in $\End(M)/J(\End(M))$ lift to idempotents in $\End(M)$. Moreover, $M$ admits a decomposition $M=N\oplus L$ such that:
\begin{enumerate}
\item[(i)] $\End(N)/J(\End(N))$ is a Boolean ring.
\item[(ii)] Each element of $\End(L)$ is the sum of two units and consequently, $L$ is $\mathcal X$-endomorphism invariant (resp., $\mathcal X$-endomorphism coinvariant).
\item[(iii)] Both $\H_R(N,L)$ and $\H_R(L,N)$ are contained in $J(\End(M))$.
\end{enumerate}
In particular, $\End(M)/J(\End(M))$ is the direct product of a Boolean ring and a right self-injective von Neumann regular ring.
\end{theorem}

\noindent Recall that a module $M$ is called {\it pseudo-injective} if given any submodule $A$ of $M$, any monomorphism $f:A\rightarrow M$ can be extended to an endomorphism of $M$ (see \cite{JS,Teply}). We will first give a new proof showing that automorphism-invariant modules coincide with pseudo-injective modules. 

Throughout this section, we will follow notations as in \cite{GKS} to write elements in direct product of rings $R_1\times R_2$ as $a_1\times a_2$ where $a_1\in R_1$ and $a_2\in R_2$. To write an element in $R/J(R)$, we will use the notation $\bar{a}$ where $a\in R$.

\begin{theorem} \label{auto} \cite{ESS}
A module $M$ is automorphism-invariant if and only if it is pseudo-injective.  
\end{theorem}

\begin{proof}
Let $M$ be an automorphism-invariant module. Let $N$ be a submodule of $M$ and $f:N\rightarrow M$ be a monomorphism. Call $E=E(M)$. It is well-known that $\End(E)/J(\End(E))$ is von Neumann regular, right self-injective and idempotents lift modulo $J(\End(E))$. Since injective modules are obviously automorphism-invariant, from Theorem \ref{struct}, we have that $E=E_1\oplus E_2$, where $\H(E_1, E_2), \H(E_2, E_1) \subseteq J(\End(E))$ and $$\End(E)/J(\End(E))=\End(E_1)/J(\End(E_1)) \times \End(E_2)/J(\End(E_2))$$ such that $\End(E_1)/J(\End(E_1))$ is a Boolean ring and each element in the ring $\End(E_2)/J(\End(E_2))$ is the sum of two units. Call $S=\End(E)$, $S_1=\End(E_1)$ and $S_2=\End(E_2)$. Let $v:E(N)\rightarrow E$ be the inclusion and $p:E\rightarrow E(N)$, an epimorphism that splits $v$. Then $e=v\circ p\in S$ is an idempotent such that $E(N)=eE$. 
By injectivity, $f$ extends to a monomorphism $g:E(N)\rightarrow E$. This monomorphism $g$ splits as $E(N)$ is injective. So there exists an epimorphism $\delta: E\rightarrow E(N)$ such that $\delta \circ g=1_{E(N)}$. Call $h=g\circ p \in S$. We claim that $\bar{h}|_{\bar{e}\bar{S}}:\bar{e}\bar{S}\rightarrow \bar{S}$ is a monomorphism. 
Let $x\in S$ such that $\bar{e}\bar{x}\in \ker(\bar{h})$. This means that $h\circ e\circ x=g\circ p\circ e\circ x$ has essential kernel. 
And, as $g$ is a monomorphism, we deduce that $p\circ e \circ x = p\circ v\circ p\circ x= p\circ x$ also has essential kernel. So, $e\circ x = v\circ p\circ x$ has essential kernel too. This shows that $\bar{e}\bar{x}=0$ and thus, $\bar{h}|_{\bar{e} \bar{S}}$ is monic. 

As $\bar{S}=\bar{S_1} \times \bar{S_2}$, there exist idempotents $\bar{e_1} \in \bar{S_1}$ and $\bar{e_2} \in \bar{S_2}$ such that $\bar{e}=\bar{e_1} \times \bar{e_2} \in \bar{S_1} \times \bar{S_2}$ and homomorphisms $\bar{h_1}: \bar{S_1} \rightarrow \bar{S_1}$ and $\bar{h_2}: \bar{S_2}\rightarrow \bar{S_2}$ such that $\bar{h}=\bar{h_1} \times \bar{h_2}$. Moreover, $\bar{h_i}|_{\bar{e_i} \bar{S_i}}: \bar{e_i}\bar{S_i} \rightarrow \bar{S_i}$ is a monomorphism and $\bar{h_i}|_{(1-\bar{e_i})\bar{S_i}}=0$ for $i=1, 2$. As $\Im(\bar{h})\cong \bar{e} \bar{S}$, it is a direct summand of $\bar{S}$. So there exists an idempotent $\bar{e'}\in \bar{S}$ such that $\Im(\bar{h})=\bar{e'} \bar{S}$. And again, $\bar{e'}=\bar{e'_1} \times \bar{e'_2}$ for idempotents $\bar{e'_1} \in \bar{S_1}$ and $\bar{e'_2} \in \bar{S_2}$. Also, we have $\Ker(\bar{h_1})=(1-\bar{e_1})S$ as $\bar{e_1}\in \bar{S_1}$ is central and $(1-e)h=0$. This yields $\bar{e_1}=\bar{e'_1}$. 

 Call $\bar{h'_1}: \bar{S_1} \rightarrow \bar{S_1}$ the homomorphism defined by $\bar{h'_1}|_{\bar{e_1} \bar{S_1}}=\bar{h_1}|_{\bar{e_1} \bar{S_1}}$ and $\bar{h'_1}|_{(1-\bar{e_1})\bar{S_1}}=1_{(1-\bar{e_1})\bar{S_1}}$. By construction, $\bar{h'_1}$ is an automorphism in $\bar{S_1}$. On the other hand, $\bar{h_2}\in \bar{S_2}$, so we can write $\bar{h_2}$ as the sum of two automorphisms, say $\bar{h_2}=\bar{h'_2} + \bar{h''_2}$. And again, $\bar{h''_2}$ can be written as the sum of two automorphisms in $\bar{S_2}$, say $\bar{h''_2}=\bar{t_2} + \bar{t'_2}$.

Set $\bar{\gamma_1}= \bar{h'_1}\times \bar{h'_2}$, $\bar{\gamma_2}= \bar{h'_1}\times \bar{t_2}$, and $\bar{\gamma_3}= (\bar{-h'_1})\times \bar{t'_2}$. Consider then the homomorphism $\bar{\gamma}=\bar{\gamma_1}+\bar{\gamma_2}+\bar{\gamma_3}$. Then $\bar{\gamma}$ is the sum of three automorphisms $\bar{\gamma_1}$, $\bar{\gamma_2}$ and $\bar{\gamma_3}$ in $\bar{S}$. Note that for any $x_1\times x_2\in \bar{e}\bar{S}=\bar{e_1} \bar{S} \times \bar{e_2} \bar{S}$, we have that $(\bar{h'_1}\times \bar{h'_2}+\bar{h'_1}\times \bar{t_2} + (-\bar{h'_1}) \times \bar{t'_2})(x_1\times x_2)=\bar{h'_1}(x_1)\times \bar{h_2}(x_2)= \bar{h_1}(x_1)\times \bar{h_2}(x_2)=\bar{h}(x_1\times x_2)$ since $\bar{h'_1}|_{\bar{e_1}\bar{S}}= \bar{h_1}|_{\bar{e_1}\bar{S}}$. This means that $\bar{\gamma}$ is the sum of three automorphisms in $\bar{S}$ and $\bar{\gamma}|_{\bar{e}\bar{S}}=\bar{h}|_{\bar{e}\bar{S}}$. Let us lift the three automorphisms $\bar{\gamma_i} \in \bar{S}$ to automorphisms $\gamma_i \in S$. As $M$ is automorphism-invariant, $\gamma_i(M)\subseteq M$ for $i=1, 2, 3$. So if we call $\gamma=\gamma_1+\gamma_2+\gamma_3$, we get that $\gamma(M)\subseteq M$. Moreover, as $\bar{\gamma}|_{\bar{e}\bar{S}}= \bar{h}|_{\bar{e}\bar{S}}$, there exists a $j\in J(S)$ such that $\gamma|_{eS}=h|_{eS}+j|_{eS}$. Thus $h|_{eS}=(\gamma-j)|_{eS}$. As $j\in J(S)$, $1-j$ is an automorphism and consequently, $M$ is invariant under $1-j$ and hence under $j$. We have already seen that $M$ is invariant under $\gamma$. Thus it follows that $M$ is invariant under $\gamma-j$. Call $\varphi=(\gamma-j)|_{M}$. Then $\varphi$ is an endomorphism of $S$ such that $\varphi|_{E(N)}=h|_{E(N)}$ and thus $\varphi|_{N}=f$. This shows that $\varphi$ extends $f$ and hence $M$ is pseudo-injective. 

The converse is straightforward (see \cite{LZ}).        
\end{proof}

\bigskip

\noindent Now, we proceed to dualize this result for automorphism-coinvariant modules.

\bigskip

\noindent Recall that a module $M$ is called a {\it pseudo-projective module} if for every submodule $N$ of $M$, any epimorphism $\varphi: M\rightarrow M/N$ can be lifted to a homomorphism $\psi: M\rightarrow M$ (see \cite{SS}). These modules generalize the class of projective modules and quasi-projective modules (see \cite{Golan, K, WJ}). It is known that any pseudo-projective module with a projective cover is dual automorphism-invariant (see \cite{SS}) and hence it is automorphism-coinvariant.

\begin{theorem} \label{dual}
If $R$ is a right perfect ring, then a right $R$-module $M$ is automorphism-coinvariant if and only if it is pseudo-projective.
\end{theorem}

\begin{proof}
Let $M$ be an automorphism-coinvariant module over a right perfect ring $R$ with a projective cover $p:P\rightarrow M$. Let $N$ be a submodule of $M$ and $f:M\rightarrow M/N$, an epimorphism. It is known that $\End(P)/J(\End(P))$ is von Neumann regular, right self-injective and idempotents lift modulo $J(\End(P))$. As projective modules are clearly automorphism-coinvariant, by Theorem \ref{struct}, we have that $P=P_1\oplus P_2$, where $\H(P_1, P_2), \H(P_2, P_1) \subseteq J(\End(P))$ and 
$$\End(P)/J(\End(P))=\End(P_1)/J(\End(P_1)) \times
\End(P_2)/J(\End(P_2))$$
such that $\End(P_1)/J(\End(P_1))$ is a Boolean ring and each element in the ring $\End(P_2)/J(\End(P_2))$ is the sum of two units.

Let us denote by $\pi: M \rightarrow M/N$ the canonical projection and call $S=\End(P)$, $S_1=\End(P_1)$ and $S_2=\End(P_2)$. As $p:P\rightarrow M$ is a projective cover, there exists a direct summand $P(M/N)$ of $P$ such that, if we denote by $v:P(M/N)\rightarrow P$ and $q:P\rightarrow P(M/N)$ the structural injection and projection respectively, then $\pi\circ p\circ v:P(M/N)\rightarrow M/N$ is the projective cover of $M/N$, $e=v\circ q$ is an idempotent in $S$, and $P(M/N)=eP$.

By projectivity, $f$ lifts to an epimorphism $g:P\rightarrow P(M/N)$ such that $(\pi\circ p\circ v)\circ g=f\circ p$. This epimorphism $g$ splits as $P(M/N)$ is projective. Thus, there exists a monomorphism $\delta: P(M/N)\rightarrow P$ such that $g\circ \delta =1_{P(M/N)}$. Call $h=v\circ g\in S$. Note that $e\circ h=h$ and so $hS\subseteq eS$. Moreover, $h\circ \delta \circ q=v\circ g\circ \delta \circ q=v\circ q=e$. So $eS\subseteq hS$ and consequently, $hS=eS$. This shows that $h:S \rightarrow eS$ is epic and consequently, $\bar{h}:\bar{S}\rightarrow \bar{e}\bar{S}$ is an epimorphism.

 As $\bar{S}=\bar{S_1} \times \bar{S_2}$, there exist idempotents $\bar{e_1} \in \bar{S_1}$ and $\bar{e_2} \in \bar{S_2}$ such that $\bar{e}=\bar{e_1} \times \bar{e_2} \in \bar{S_1} \times \bar{S_2}$ and homomorphisms $\bar{h_1}: \bar{S_1} \rightarrow \bar{S_1}$ and $\bar{h_2}: \bar{S_2}\rightarrow \bar{S_2}$ such that $\bar{h}=\bar{h_1} \times \bar{h_2}$.  

 Moreover, $\bar{e_i}\circ \bar{h_i}: \bar{S_i}\rightarrow \bar{e_i}\bar{S_i}$ is an epimorphism, and $(1-\bar{e_i})\circ \bar{h_i}=0$ for $i=1, 2$. As $\Im(\bar{h})\cong \bar{e} \bar{S}$, it is a direct summand of $\bar{S}$. So there exists an idempotent $\bar{e'}\in \bar{S}$ such that $\Im(\bar{h})=\bar{e} \bar{S}$ and $\Ker(\bar{h})=(\bar{1}-\bar{e'})\bar{S}$. And again, $\bar{e'}=\bar{e'_1} \times \bar{e'_2}$ for idempotents $\bar{e'_1} \in \bar{S_1}$ and $\bar{e'_2} \in \bar{S_2}$. Also, we have $\Ker(\bar{h_1})=(1-\bar{e_1})S$ as $\bar{e_1}\in \bar{S_1}$ is central and $(1-e)h=0$. This yields $\bar{e_1}=\bar{e'_1}$. 

 Call $\bar{h'_1}: \bar{S_1} \rightarrow \bar{S_1}$ the homomorphism defined by $\bar{h'_1}|_{\bar{e_1} \bar{S_1}}=\bar{h_1}|_{\bar{e_1} \bar{S_1}}$ and $\bar{h'_1}|_{(1-\bar{e_1})\bar{S_1}}=1_{(1-\bar{e_1})\bar{S_1}}$. By construction, $\bar{h'_1}$ is an automorphism in $\bar{S_1}$. On the other hand, $\bar{h_2}\in \bar{S_2}$, so we can write $\bar{h_2}$ as the sum of two automorphisms, say $\bar{h_2}=\bar{h'_2} + \bar{h''_2}$. And again, $\bar{h''_2}$ can be written as the sum of two automorphisms in $\bar{S_2}$, say $\bar{h''_2}=\bar{t_2} + \bar{t'_2}$. 
 
 Set $\bar{\gamma_1}= \bar{h'_1}\times \bar{h'_2}$, $\bar{\gamma_2}= \bar{h'_1}\times \bar{t_2}$, and $\bar{\gamma_3}= (\bar{-h'_1})\times \bar{t'_2}$. Consider then the homomorphism $\bar{\gamma}=\bar{\gamma_1}+\bar{\gamma_2}+\bar{\gamma_3}$. Then $\bar{\gamma}$ is the sum of three automorphisms $\bar{\gamma_1}$, $\bar{\gamma_2}$ and $\bar{\gamma_3}$ in $\bar{S}$. Note that for any $x_1\times x_2\in \bar{e}\bar{S}=\bar{e_1} \bar{S} \times \bar{e_2} \bar{S}$, we have that $(\bar{h'_1}\times \bar{h'_2}+\bar{h'_1}\times \bar{t_2} + (-\bar{h'_1}) \times \bar{t'_2})(x_1\times x_2)=\bar{h'_1}(x_1)\times \bar{h_2}(x_2)= \bar{h_1}(x_1)\times \bar{h_2}(x_2)=\bar{h}(x_1\times x_2)$ since $\bar{h'_1}|_{\bar{e_1}\bar{S}}= \bar{h_1}|_{\bar{e_1}\bar{S}}$. This means that $\bar{\gamma}$ is the sum of three automorphisms in $\bar{S}$ and $\bar{\gamma}|_{\bar{e}\bar{S}}=\bar{h}|_{\bar{e}\bar{S}}$. Let us lift the three automorphisms $\bar{\gamma_i} \in \bar{S}$ to automorphisms $\gamma_i \in S$. As $M$ is automorphism-coinvariant, $M$ is coinvariant under $\gamma_i$ for $i=1, 2, 3$. So if we call $\gamma=\gamma_1+\gamma_2+\gamma_3$, we get that $M$ is coinvariant under $\gamma$. Moreover, as $\bar{e}\circ \bar{\gamma}= \bar{e}\circ \bar{h}$, there exists a $j\in J(S)$ such that $\pi \circ p\circ \gamma=\pi \circ p\circ h+\pi \circ p\circ j$. Thus $\pi \circ p\circ h=\pi \circ p\circ (\gamma-j)$. As $j\in J(S)$, $1-j$ is an automorphism and consequently, $M$ is coinvariant under $1-j$ and hence under $j$. We have already seen that $M$ is coinvariant under $\gamma$. Thus $M$ is coinvariant under $\gamma-j$ and hence by definition, it follows that there exists an endomorphism $t:M\rightarrow M$ such that $p\circ (\gamma-j)=t\circ p$. This gives $\pi \circ p\circ h=\pi\circ t\circ p$. As $\pi \circ p\circ h=f\circ p$, we have $\pi \circ t \circ p=f\circ p$. As $p$ is epic, we have $\pi\circ t=f$. 
 
 Thus $t$ is an endomorphism of $M$ such that $\pi\circ t=f$. This shows that  $M$ is pseudo-projective. 

The converse is straightforward \cite{SS}.   
\end{proof} 

\begin{remark} Note that the arguments given in the proof of Theorem\,\ref{dual} can be easily adapted to the situation in which the ring $R$ is only assumed to be semiperfect instead of right perfect, and the module $M$ is finitely generated and therefore, it has a projective cover.  
\end{remark}

\bigskip

\noindent {\bf Acknowledgment.} We would like to thank the referee for his/her helpful comments which motivated us to improve the presentation of this paper.

\end{document}